\documentclass[a4paper,11pt]{article}
\usepackage{latexsym}
\usepackage{amssymb}
\usepackage{enumerate}
\usepackage{lscape}
\usepackage[shortlabels]{enumitem}
\usepackage{xcolor}
\usepackage{commath}
\usepackage{comment}
\usepackage{amsfonts}
\usepackage{amsmath,amsthm}
\usepackage{hyperref}
\usepackage{algorithm}
\usepackage{algorithmic}
\usepackage{relsize}
\usepackage[title]{appendix}
\usepackage{framed}
\usepackage{boites}
\usepackage[pdftex]{graphicx}
\newtheorem{mainthm}{Theorem}

\newtheorem{thm}{Theorem}

\newenvironment{@abssec}[1]{%
    \if@twocolumn

      \section*{#1}%
    \else

      \vspace{.05in}\footnotesize
      \parindent .2in
 {\upshape\bfseries #1. }\ignorespaces
    \fi}

    {\if@twocolumn\else\par\vspace{.1in}\fi}

\newenvironment{keywords}{\begin{@abssec}{\keywordsname}}{\end{@abssec}}

\newenvironment{AMS}{\begin{@abssec}{\AMSname}}{\end{@abssec}}

\newcommand\keywordsname{Key words}
\newcommand\AMSname{AMS subject classifications}
\newcommand\AMname{AMS subject classification}
\newcommand\restr[2]{{% we make the whole thing an ordinary symbol
\left.\kern-\nulldelimiterspace % automatically resize the bar with \right
#1 % the function
\vphantom{|} % pretend it's a little taller at normal size
\right|_{#2} % this is the delimiter
}}
\newtheorem{theorem}{Theorem}[section]
\newtheorem{lemma}[theorem]{Lemma}
\newtheorem{corollary}[theorem]{Corollary}

\newtheorem{remark}[theorem]{Remark}
\newtheorem{definition}[theorem]{Definition}

\makeatletter
\def\moverlay{\mathpalette\mov@rlay}
\def\mov@rlay#1#2{\leavevmode\vtop{%
   \baselineskip\z@skip \lineskiplimit-\maxdimen
   \ialign{\hfil$\m@th#1##$\hfil\cr#2\crcr}}}
\newcommand{\charfusion}[3][\mathord]{
    #1{\ifx#1\mathop\vphantom{#2}\fi
        \mathpalette\mov@rlay{#2\cr#3}
      }
    \ifx#1\mathop\expandafter\displaylimits\fi}
\makeatother

\newcommand{\cupdot}{\charfusion[\mathbin]{\cup}{\cdot}}

\def\XXint#1#2#3{{\setbox0=\hbox{$#1{#2#3}{\int}$}
\vcenter{\hbox{$#2#3$}}\kern-.5\wd0}}

\newcommand{\link}{\mathop{\circ\kern-.35em -}}

\newcommand{\pa}{\partial}

\newcommand{\dv}{\mathop{\mathrm{div}}}

\newcommand{\al}{\alpha}
\newcommand{\be}{\beta}

\newcommand{\si}{\sigma}
\newcommand{\Si}{\Sigma}

\newcommand{\Om}{\Omega}

\newcommand\setbld[2]{\left\{ #1 \;\middle |\; #2\right\}}

\newcommand{\cdottone}{{\boldsymbol{\cdot}}}

\newcommand{\identmatrix}%{I_N}
{{\mathbf{I}}}

\numberwithin{equation}{section}

\title{Elliptic and parabolic overdetermined problems in multi-phase settings}

\author{Lorenzo Cavallina \, Giorgio Poggesi \
}
\date{}

\begin{document}

\maketitle

\begin{abstract}
The present paper provides symmetry results for a class of overdetermined problems of elliptic and parabolic type in multi-phase settings, including various extensions of remarkable results obtained by S. Sakaguchi in \cite{sakaguchi2016rendiconti, sakaguchibessatsu}. A new alternative approach to proving this type of results is presented, leveraging the weak formulation of the problem. The resulting proofs are direct and elegant, and bring several benefits, including extensions to multi-phase settings (possibly with infinitely many phases) and generalizations to rough interfaces. 
\end{abstract}


\maketitle

\pagestyle{plain}
\thispagestyle{plain}

\section{Introduction}\label{sec:introduction}
The present paper provides symmetry results for a class of overdetermined problems of elliptic and parabolic type in multi-phase settings.
The foundational result concerning overdetermined problems was famously obtained by J. Serrin in 1971 in the seminal paper \cite{serrin1971} and dealt with problems of elliptic type in the one-phase setting. Its simplest formulation states that, given a $C^2$ bounded domain $\Omega \subset \mathbb{R}^N$, the {\it torsion problem}
\begin{equation*}
    \begin{cases}
-\Delta u = 1\quad \text{in }\Omega,\\
u=0\quad \text{on }\partial\Omega,\\
    \end{cases}
\end{equation*}
coupled with the overdetermined condition
\begin{equation}\label{eq:Serrin over}
    \partial_\nu u= c \quad \text{on }\partial\Omega ,
\end{equation}
where $\partial_\nu$ denotes the normal derivative and $c$ is a constant, admits a solution $u \in C^2(\overline{\Omega})$ if and only if $\Omega$ is a ball.

The {\it two-phase setting} essentially amounts to replacing the classical Laplacian with an operator with discontinuous coefficients; more precisely, given a bounded domain $\Omega\subset \mathbb{R}^N$, an open set $D$ such that $\overline{D} \subset \Omega$, and a piecewise constant function
\begin{equation*}
    \si(x) =
    \begin{cases}
        \si_c>0 \quad \text{if } x\in D , \\
        \si_s>0 \quad \text{if } x\in \Om \setminus D ,
    \end{cases}
\end{equation*}
the {\it two-phase torsion problem} is defined as
\begin{equation}\label{eq:2 phase pb}
    \begin{cases}
-\dv(\sigma\nabla u)=1\quad \text{in }\Omega,\\
u=0\quad \text{on }\partial\Omega .
    \end{cases}
\end{equation}
In this context, the natural question is whether \eqref{eq:2 phase pb} coupled with the Serrin-type overdetermination \eqref{eq:Serrin over} may lead to a symmetry result, as it happens in the classical (i.e., one-phase) setting. Such a question has a negative answer, and, as shown in \cite{cavallina&yachimura2020two}, there exist infinitely many non-symmetric solutions to \eqref{eq:2 phase pb}--\eqref{eq:Serrin over}. A rough but intuitive explanation in this case, we are dealing with only $1$ overdetermination (namely, \eqref{eq:Serrin over}), which is strictly less than the $2$ degrees of freedom brought by $D$ and $\Omega$ (both $D$ and $\Omega$ are unknowns). Hence, a natural question is ``how much overdetermination is enough to get symmetry in two-phase problems?'' A comprehensive answer and a full characterization of both symmetry and asymmetry results in a multi-phase setting has been recently obtained by the authors in \cite{face2phase}.
On another direction, one may wonder whether reducing the degrees of freedom brought by $D$ and $\Omega$ may lead to symmetry with only $1$ Serrin-type overdetermination. For instance, one may set either $D$ or $\Om$ to be a ball. In the first case (i.e., $D$ is a ball), the answer is negative and, in fact, as shown in \cite{cavallina&yachimura2021symmetry}, even when $D$ is a ball, there exist infinitely many non-symmetric solutions to \eqref{eq:2 phase pb}--\eqref{eq:Serrin over}. In the second case (i.e., $\Omega$ is a ball), an affirmative answer was obtained by Sakaguchi in \cite{sakaguchi2016rendiconti, sakaguchibessatsu} and reads as follows.
\begin{thm}[{\cite[Theorem 5.1]{sakaguchibessatsu}}]\label{thm:Sakaguchi}
Let $\Omega$ be a ball and let $\emptyset \neq D \subset \subset \Om$ be the disjoint union of finitely many domains of class $C^2$. Assume that $\Omega \setminus D$ is connected. Then, the overdetermined problem \eqref{eq:2 phase pb}--\eqref{eq:Serrin over} admits a solution if and only if $(D,\Om)$ are concentric balls.
\end{thm}
\begin{remark}
    The same conclusion also holds in the case where the overdetermination $\partial_\nu u = c$ is assumed just on a connected component $\Sigma$ of $\partial \Omega$ and $\Sigma$ is a sphere. This is not stated in \cite{sakaguchibessatsu}, but it is essentially shown in \cite{sakaguchi2016rendiconti}.
\end{remark}
We recall that when $D$ is of class $C^2$ (as in the above result), \eqref{eq:2 phase pb} can be rewritten as a {\it transmission problem}. Sakaguchi's proof consists of a careful and detailed case-by-case analysis, where the {\it transmission conditions} are cleverly used to get a contradiction to Hopf's lemma.

In the present paper, we provide an alternative approach to this type of results. Our proofs are more direct and leverage the weak formulation of the problem. This brings several benefits. First, it allows us to extend the result to general multi-phase settings, possibly including infinitely many phases. Second, it allows us to generalize the result to the case of rough interfaces (i.e., with no regularity assumptions on $D$).
We finally mention that for simplicity we stated Theorem \ref{thm:Sakaguchi} for the equation $-\dv(\sigma\nabla u)=1$, but Sakaguchi's approach allows to treat equations of the type $-\dv(\sigma\nabla u)= \be-\al u$, where $\al \ge 0$ and $\be>0$ are real constants (see \cite[Theorem 5.1]{sakaguchibessatsu}); our approach instead allows to work with equations of the type $-\dv(\sigma\nabla u)= g$, where $g\in L^2(\Omega)$ is a positive radial function.

\subsection{The main result concerning an elliptic overdetermined problem}

In what follows, we will explain the setting and the notation employed. 
\begin{equation}\label{Omega}
  \text{Let $\Omega$ denote a bounded domain in $\mathbb{R}^N$ ($N\ge 2$).
}  
\end{equation}
\begin{equation}\label{partial Omega}
\text{Let $\partial\Omega$ be made entirely of regular points for the Dirichlet Laplacian}   
\end{equation}
(concerning the precise definition of ``regular point", we refer the interested reader to \cite[Chapter 8]{GilbargTrudinger}, where the well-known Wiener criterion is also discussed). Also, 
\begin{equation}\label{Gamma}
\text{let $\Gamma\subset\Omega$ be a closed set with Lebesgue measure zero.}    
\end{equation}
Let $\sigma_m$ and $\sigma_M$ be two positive real constants such that $0<\sigma_m<1<\sigma_M$ and
\begin{equation}
\text{let }\sigma: \Omega\setminus \Gamma\to [\sigma_m,\sigma_M] \text{ be a locally constant function}.
\end{equation}
We recall that a function is ``locally constant" if and only if it is constant on each connected component of its domain. 
Throughout this paper, let us assume that 
\begin{equation}\label{S}
\text{there exists a unique connected component $S$ of $\Omega\setminus\Gamma$ such that $\partial\Omega\subset\overline{S}$}  
\end{equation}
and we will set 
\begin{equation}\label{D}
D:=(\Omega\setminus\Gamma)\setminus S.
\end{equation}
 We will refer to $S$ and $D$ as ``\emph{shell}" and ``\emph{core}" respectively. Notice that, by construction, the open set $\Omega\setminus\Gamma$ has at most countably many connected components (indeed, for each such connected component, one can find a distinct point with only rational coordinates). It follows that the function $\sigma$ can take at most countably many values. For simplicity, we will assume that 
 \begin{equation}\label{sigmas}
    S=\sigma^{-1}\left(\{1\}\right).
    \end{equation}
Moreover, we will let $\{\sigma_i\}_{i\in I}\subset\mathbb{R}\setminus\{1\}$ denote the collection of all the values that $\sigma$ takes in $D$, indexed by $I\subseteq\mathbb{N}$. 

We will write:
\begin{equation}\label{D_i}
    D=\bigcup_{i\in I} D_i, \quad D_i:=\sigma^{-1}\left(\{\sigma_i\}\right).
\end{equation}
Throughout this paper, we will assume that the collection of all phases in the core is discrete. In other words,  
\begin{equation}\label{discreteness}
    \text{for all } x\in\Omega,\; \text{there exists } \varepsilon>0\; \text{such that } 
    \left| 
    \setbld{i\in I}{\overline{D_i}\cap B_\varepsilon(x)\ne \emptyset}
    \right|
    \le 1 .
\end{equation}
See Figure \ref{fig: discreteness} for a geometrical interpretation of such an assumption.
Here and throughout the paper, we denote by $B_r (x)$ the ball of radius $r>0$ centered at $x \in \mathbb{R}^N$; the ball of radius $r>0$ centered at the origin will be simply denoted by $B_r$.
\begin{figure}
    \centering
    \includegraphics[width=0.95\linewidth]{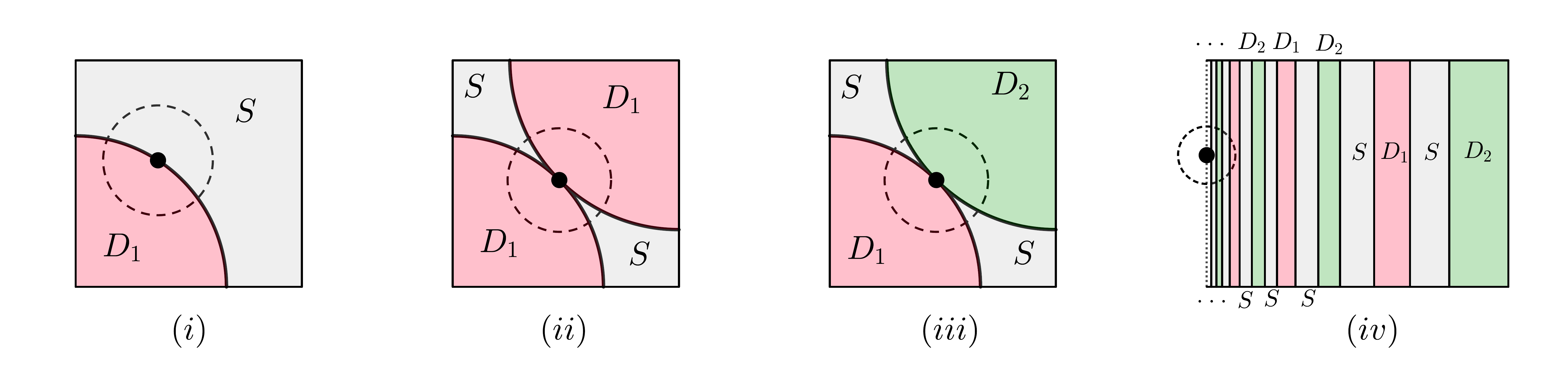}
    \caption{The geometrical interpretation of the discreteness assumption \eqref{discreteness}. Here examples $(i)$ and $(ii)$ satisfy \eqref{discreteness}, while examples $(iii)$ and $(iv)$ do not. Notice that in $(iii)$, two distinct phases of the core touch, thus violating \eqref{discreteness}. On the other hand, $(iv)$ still violates \eqref{discreteness} without two distinct phases of the core touching.} 
    \label{fig: discreteness}
\end{figure}

In some of the results, we will assume that 
\begin{equation}\label{Sigma}
\text{the boundary $\partial\Omega$ contains a connected component $\Sigma$, which is a sphere.}    
\end{equation}
Whenever we do so, without loss of generality, we will assume that the sphere is centered at the origin. Specifically, we denote it by
\begin{equation*}
    \Sigma=\partial B_R ,
\end{equation*}
where $B_R$ is the ball of radius $R>0$ centered at the origin.

\begin{figure}
    \centering
    \includegraphics[width=0.7\linewidth]{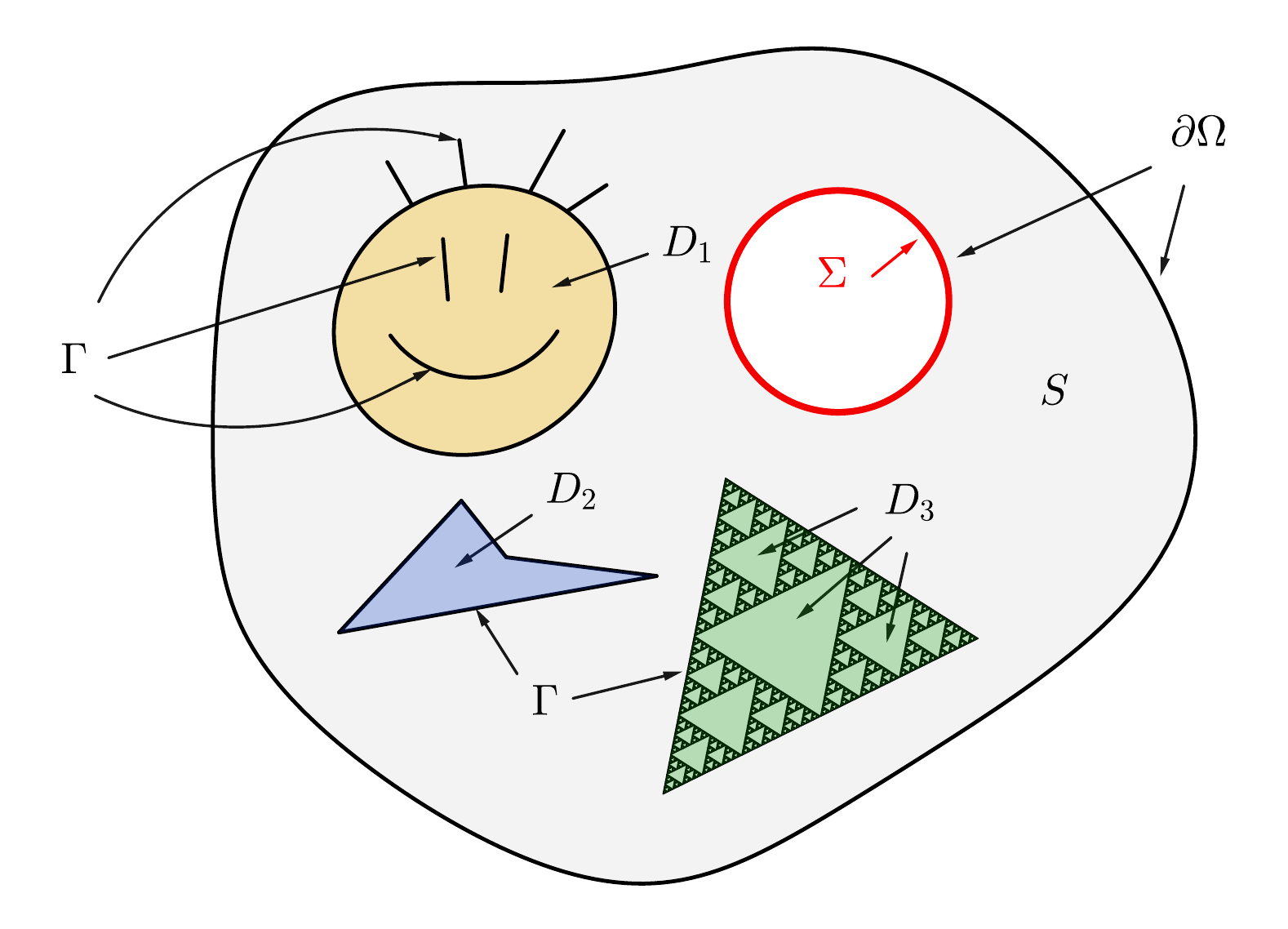}
    \caption{The geometric setting of our problem. Notice that we do not impose any regularity assumption on $\Gamma$, allowing even \emph{fractal interfaces}.}
    \label{fig: problem setting}
\end{figure}

Finally, 
\begin{equation}\label{g is radial}
\text{let 
$g\in L^2(\Omega)$ be a positive radial function}
\end{equation} (that is, $g$ admits a representative whose value at $x$ is positive and only depends on $|x|$)
such that
\begin{equation}\label{g is analytic}
\text{
its restriction $\restr{g}{S}$ is a real analytic function.} 
\end{equation}
 
We consider the following overdetermined problem:
\begin{equation}\label{odp}
    \begin{cases}
-\dv(\sigma\nabla u)=g\quad \text{in }\Omega,\\
u=0\quad \text{on }\partial\Omega,\\
\partial_\nu u= c\quad \text{on }\Sigma,
    \end{cases}
\end{equation}
where $\partial_\nu$ denotes the normal derivative, computed with respect to the outward unit normal $\nu$ and $c$ is a constant. 

In particular, we say that the overdetermined problem \eqref{odp} admits a solution if there exists a function $u\in H_0^1(\Omega)$ that satisfies both 
\begin{equation}\label{weak form}
    \int_\Omega  \sigma \nabla u \cdot \nabla \varphi = \int_\Omega g\varphi \quad \text{for all }\varphi\in H_0^1(\Omega),
\end{equation}
and 
\begin{equation}\label{overdetermined condition}
\partial_\nu u =c \quad \text{on }\Sigma.
\end{equation}
Notice that we can require the overdetermined condition \eqref{overdetermined condition} to hold pointwise on $\Sigma$ since the unique solution $u$ to \eqref{weak form} is arbitrarily smooth in a neighborhood of $\Sigma$ by the standard Schauder boundary estimates for the Poisson equation (see, for instance, \cite[Lemma 6.18]{GilbargTrudinger} and the subsequent discussion).

\begin{mainthm}\label{mainthm I}
In the setting given by \eqref{Omega}--\eqref{g is analytic}, the following are equivalent.
\begin{enumerate}[(i)]
    \item The overdetermined problem \eqref{odp} admits a solution $u\in H_0^1(\Omega)$.
    \item 
    \begin{itemize}
        \item Either $I=\emptyset$ and $\Omega$ is either a ball or a spherical annulus (one-phase), 
        \item or $I=\{i\}$ and $(\overline{D}, \overline{\Omega})=(\overline{D_i}, \overline{\Omega})$ are  concentric closed balls (two-phase).
    \end{itemize}
\end{enumerate}
\end{mainthm}

\begin{remark}
Notice that we did not impose any regularity assumption on $D$, but only a discreteness-type assumption, namely \eqref{discreteness}, ensuring that no two distinct phases in the core ``touch" (see Figures \ref{fig: discreteness}--\ref{fig: problem setting}). We remark that \eqref{discreteness} is trivially satisfied in the two-phase setting, as in such a setting the core is made up of one phase only. Thus, in the two-phase setting, the symmetry result holds for any open set $D$ such that $\overline{D} \subset \Omega$ and $S= \Omega \setminus \overline{D}$ is connected.
We stress that the assumption that $S$ is connected, that is \eqref{S}, cannot be removed (see \cite[Section $6$]{cavallina2024symmetry&asymmetry} where counterexamples are shown for disconnected~$S$).
\end{remark}

\subsection{Applications to the parabolic setting}

Leveraging Theorem \ref{mainthm I}, we will also obtain symmetry results concerning multi-phase overdetermined parabolic problems. 

We consider the following initial/boundary-value problem:
\begin{equation}\label{parabolic_bvp}
    \begin{cases}
\partial_t u = \dv(\sigma\nabla u) \quad \text{in }\Omega \times (0,\infty),\\
u=1 \quad \text{in } \Omega \times \left\{ 0 \right\} , \\
u = 0\quad \text{on }\partial\Omega \times (0,\infty) ,\\
    \end{cases}
\end{equation}
where $u=u(x,t)$ and $\partial_t$ denotes the partial derivative with respect to $t$.

In the setting given by \eqref{Omega}--\eqref{discreteness}, let $G$ be a bounded domain such that
\begin{equation*}
    \overline{D} \subset G  , \quad \overline{G} \subset \Omega , \quad \text{and} \quad \mathrm{dist}(x,\overline{D}) \ge \mathrm{dist}(x,\partial \Omega)\quad \text{for all }x\in \partial G.
\end{equation*}
We consider one of the following overdetermined conditions:
\begin{equation}\label{par_over_isothermic}
    u(x,t) = a(t) \quad \text{on } M \times (0,\infty) ,
\end{equation}
\begin{equation}\label{par_over_constant flow G}
     \partial_\nu u (x,t) = b(t) \quad \text{on } M \times (0,\infty) , \end{equation}
where $M$ is a connected component of $\partial G$ and $a,b:(0,\infty)\to\mathbb{R}$ are real valued functions.
From a physical point of view, \eqref{par_over_isothermic} says that $M$ is an isothermic surface, while \eqref{par_over_constant flow G} says that $M$ is a surface with the constant flow property (see Section \ref{section Parabolic} for details). 
In the two-phase setting with smooth interfaces, symmetry results for \eqref{parabolic_bvp}--\eqref{par_over_isothermic} and \eqref{parabolic_bvp}--\eqref{par_over_constant flow G} were obtained in \cite{sakaguchi2016rendiconti, sakaguchibessatsu, CavallinaMagnaniniSakaguchi2021constantflowproperty}.
Here, we provide an alternative proof based on a corollary of Theorem \ref{mainthm I} (see Corollary \ref{cor_parabolic_1}), which allows to extend the results to the multi-phase setting with possibly infinitely many phases and rough interfaces.

\begin{mainthm}\label{thm:parabolic}
In the setting given by \eqref{Omega}--\eqref{discreteness}, the following are equivalent.
\begin{enumerate}[(i)]
    \item The weak solution $u$ of \eqref{parabolic_bvp} (see Definition~\ref{definition}) satisfies \eqref{par_over_isothermic}.
    \item The weak solution $u$ of \eqref{parabolic_bvp} (see Definition~\ref{definition}) satisfies \eqref{par_over_constant flow G} and $M$ is of class~$C^2$. 
 \item 
    \begin{itemize}
        \item Either $I=\emptyset$ and $\Omega$ is either a ball or a spherical annulus (one-phase), 
        \item or $I=\{i\}$ and $(\overline{D}, \overline{\Omega})=(\overline{D_i}, \overline{\Omega})$ are  concentric closed balls (two-phase).
    \end{itemize}
\end{enumerate}
\end{mainthm}

\subsection{Organization of the paper}

The paper is organized as follows. Section \ref{section Preliminaries} contains some preliminary results that will be crucial for the proof of Theorem \ref{mainthm I}, which is contained in Section \ref{section Proof Thm I}. Section \ref{section Parabolic} is devoted to symmetry results for multi-phase overdetermined problems of parabolic type (Corollary \ref{cor_parabolic_1} and Theorem \ref{thm:parabolic}), which will be obtained leveraging Theorem \ref{mainthm I}.

\section{Some preliminary results}\label{section Preliminaries}

In this section, we will show some technical lemmas that will be crucial to show Theorem~\ref{mainthm I}.
Throughout this section, we will work in the setting given by \eqref{Omega}--\eqref{g is analytic}.

The following result relates the solvability of \eqref{odp} with the rotational symmetry of~$\Omega$.

\begin{lemma}\label{lem: two cases}
Let the assumptions \eqref{Omega}--\eqref{S}, \eqref{g is radial}--\eqref{g is analytic} be in force. If the overdetermined problem \eqref{odp} admits a solution, then one of the following alternatives holds:
\begin{equation}\label{two possibilities}
(i)\;\; \text{$\Omega$ is a ball $\{|x|<R\}$}, \quad 
\text{or}
\quad (ii)\;\;\text{$\Omega$ is an annulus $\{R_1<|x|<R_2\}$.}    \end{equation}
\end{lemma}
\begin{proof}
We first notice that, if the overdetermined problem \eqref{odp} admits a solution $u\in H_0^1(\Omega)$, then $\Omega$ must be radially symmetric. Indeed, by the Cauchy--Kovalevskaya theorem \cite[Subsection 4.6.3]{Evans2010PartialDifferentialEquations}, the solution $u$ is radial in a neighborhood of $\Sigma$. Then, since $u$ is real analytic in the open connected component $S$ (see \cite[Theorem 9.5.1 and Definition 9.3.2 of Chapter IX]{hormander2003TheAnalysisofLinearPartialDifferentialOperatorsI}), the function $u$, which is continuous up to the boundary by \eqref{partial Omega}, must be radial up to the closure $\overline{S}\supset\partial\Omega$. Finally, since $g>0$ by \eqref{g is radial}, $u$ never vanishes inside $\Omega$ by the maximum principle, whence $\partial \Omega$ is the $0$-level set of $u$. The desired conclusion follows.
\end{proof}

In the following lemma, under the assumption that \eqref{odp} is solvable, we provide the existence of a radially symmetric auxiliary function that satisfies a related one-phase overdetermined problem.

\begin{lemma}\label{lem radial function}
Let the assumptions \eqref{Omega}--\eqref{g is analytic} be in force. If the overdetermined \eqref{odp} admits a solution $u\in H_0^1(\Omega)$, then there exists a radial function $q$ satisfying the following:
\begin{equation*}
  -\Delta q=g\quad \text{in }\Omega,\quad q=u=0 \quad \text{on }\partial\Omega, \quad \partial_\nu q= \partial_\nu u=c \quad \text{on }\partial\Omega.  
\end{equation*}
\end{lemma}
\begin{proof}
We will show that such a function $q$ exists in either case of \eqref{two possibilities}. 
First, assume case ($i$) of \eqref{two possibilities}. Consider the unique solution $q$ to 
\begin{equation*}
 -\Delta q = g \quad \text{in }B_R, \quad q=0\quad \text{on }\partial B_R.   
\end{equation*}
By unique solvability, $q$ is radial in $B_R$. In particular, $\partial_\nu q \equiv c_1$ on $\partial B_R$ for some real $c_1$. Moreover, by the divergence theorem, we get
\begin{equation*}
c_1|\partial B_R| = \int_{\partial B_R}\partial_\nu q = \int_{B_R} \Delta q = -\int_{B_R} g = \int_{B_R}\dv(\sigma \nabla u) =  \int_{\partial B_R} \partial_\nu u = c|\partial B_R|,
\end{equation*}
whence $c_1=c$. This concludes the proof in case ($i$) of \eqref{two possibilities}.

Let us now consider case ($ii$) of \eqref{two possibilities}. First of all, notice that $u$ is radial in $\overline{S}$ by construction. In other words, there exists a function $U$ such that
\begin{equation*}
   u(x)=U(|x|) \quad \text{for all } x\in \overline{S}.  
\end{equation*}
In particular, since $\Delta u=g$ in $S$, we get
\begin{equation}\label{U eq}
    U''(|x|)+\frac{N-1}{|x|}U'(|x|)= -g(x) 
\end{equation}
for all $x\in S$.
Consider now the following set:
\begin{equation*}
    J:=\setbld{|x|}{x\in\overline{S}}.
\end{equation*}
Notice that $J$ is connected, as it is the image of a connected set via a continuous function. Moreover, since $\partial\Omega\subset\overline{S}$ by assumption, we have $J=[R_1,R_2]$. In particular, we have that $\Omega=S$ and hence \eqref{U eq} holds for all $x\in \Omega$.
Set now 
\begin{equation*}
q(x):= U(|x|) \quad \text{for }x\in \overline{\Omega}.
\end{equation*}
By construction, $q\equiv u$ in $\overline{S}$, whence in particular, $q$ and $u$ share the same Cauchy data on $\partial\Omega$. Moreover, \eqref{U eq} yields that $-\Delta q = g$ in $\Omega$. This completes the proof.
\end{proof}

The following lemma exploits the discreteness assumption \eqref{discreteness} to build a smooth extension of $\restr{\sigma}{D}$ to the whole domain.
\begin{lemma}\label{lem extension} Let $\sigma$ be the locally constant function defined as in \eqref{Sigma} and let the collection $\{D_i\}_{i\in I}$ defined in \eqref{D_i} satisfy the discreteness assumption \eqref{discreteness}.
Then, there exists a smooth function $\widetilde{\sigma}\in C^\infty(\mathbb{R}^N)$ such that $\restr{\widetilde{\sigma}}{D}=\restr{\sigma}{D}$.
\end{lemma}
\begin{proof}
Without loss of generality, let us assume that $I$ can be written as 
\begin{equation*}
    \text{either } I=\{1,\dots,n\} \;(n\in \mathbb{N}), \quad \text{or }I=\mathbb{N}.
\end{equation*}
First of all, we will show that, for each $i\in I$, we can find some radius $\varepsilon_i>0$ such that 
\begin{equation}\label{disjointedness}
 \left(\overline{D_i}+B_{\varepsilon_i}\right)\cap \left(\overline{D_j}+B_{\varepsilon_j}\right)=\emptyset \quad \text{whenever }i\ne j.   
\end{equation}
To this end, for all $i\in I$, set
\begin{equation*}
    K_i:=\bigcup_{j\in I,\; j>i} \overline{D_j},
\end{equation*}
where the empty union (corresponding to the case where $I$ is finite and $i$ is the maximum of $I$) is intended to be the empty set. Notice that, since the collection $\big\{\overline{D_i}\big\}_{i\in I}$ is discrete in the sense of \eqref{discreteness}, $K_i$ is a closed set for each $i\in I$. 
In what follows, we will construct the sequence of radii $\varepsilon_i$ by induction. First, for $i=1$, there exists some radius $\varepsilon_1>0$ such that $\overline{D_1}+B_{\varepsilon_1}$ and $K+B_{\varepsilon_1}$ do not intersect.
Assume now that we have defined a finite sequence of radii $\varepsilon_1>\cdots> \varepsilon_i$. If $i$ is the maximum of $I$, we are done. Otherwise, let $\varepsilon_{i+1}>0$ be some radius (which, without loss of generality, will be taken to be smaller than the previous one, $\varepsilon_i$) such that $\overline{D_{i+1}}+B_{\varepsilon_{i+1}}$ and $K_{i+1}+B_{\varepsilon_{i+1}}$ do not intersect. Proceeding in this way, we can define a decreasing sequence of positive radii $(\varepsilon_i)_{i\in I}$ satisfying \eqref{disjointedness}.

Then, by a standard mollification argument, for any given $i\in I$, we can find a smooth function $\xi_i\in C^\infty(\mathbb{R}^N)$ satisfying the following:
\begin{equation*}
   0\le \xi_i \le \sigma_i, \quad \restr{\xi_i}{\overline{D_i}+B_{\varepsilon_i/2}}\equiv \sigma_i, \quad \restr{\xi_i}{\mathbb{R}^N\setminus\left(\overline{D_i}+B_{\varepsilon_i} \right)}\equiv 0. 
\end{equation*}
Finally, the function $\widetilde\sigma:=\sum_{i\in I} \xi_i$ has the desired property. 
\end{proof}

In what follows, let $A$ be a spherical annulus $\{R_1<|x|<R_2\}$ for $0<R_1<R_2$.
For any Sobolev function $f\in H^1(\Omega)$, we define its angular gradient $\nabla_\theta f$ as 
\begin{equation*}
    \nabla_\theta f:= \nabla f - \left(\nabla f\cdot \tfrac{x}{|x|}\right) \tfrac{x}{|x|}.  
\end{equation*}
Clearly, $\nabla_\theta f\in L^2(A,\mathbb{R}^N)$.

The following lemma shows how the angular gradient can be used to characterize the radial symmetry of functions defined on an annulus.

\begin{lemma}\label{lem radial symmetry}
Let $A$ be a spherical annulus $\{R_1<|x|<R_2\}$. A function $f\in C(A)\cap H^1(A)$ satisfies $\nabla_\theta f=0$ if and only if it is radial.    
\end{lemma}
\begin{proof}
Let $\psi\in C^\infty(\mathbb{S}^{N-1})$ and $\eta\in C^\infty_c((R_1,R_2))$ and set $\varphi(r\theta):=\psi(\theta)\eta(r)$ for $\theta\in \mathbb{S}^{N-1}$ and $r\in (R_1,R_2)$. We have
\begin{equation*}
    \begin{aligned}
0=        \int_A \nabla_\theta f\cdot \nabla \varphi = \int_A \left(\nabla f - \left(\nabla f\cdot \tfrac{x}{|x|}\right)\tfrac{x}{|x|}\right)\cdot \nabla\varphi
= -\int_A f \Delta\varphi + \int_A f \dv\left\{(\nabla\varphi\cdot\tfrac{x}{|x|})\tfrac{x}{|x|}\right\}        \\
= -\int_A f \ \big\{ \underbrace{\Delta\varphi - \left(D^2\varphi\ \tfrac{x}{|x|}\right)\cdot \tfrac{x}{|x|}  - (N-1)\nabla\varphi\cdot\tfrac{x}{|x|} }_{=\Delta_{\mathbb{S}^{N-1}}\varphi}\big\} \\
= - \int_{R_1}^{R_2}  \underbrace{\Big(\int_{\mathbb{S}^{N-1}} f(r\theta) \ \Delta_{\mathbb{S}^{N-1}}\psi(\theta) \, dS_\theta\Big)}_{:=\xi(r)} \eta(r) r^{N-1} dr ,
    \end{aligned}
\end{equation*}
where $\Delta_{\mathbb{S}^{N-1}}$ is the Laplace--Beltrami operator on the unit sphere (see \cite[Chapter 3]{atkinson_han2012spherical_harmonics_and_approximations}). 

Now, by the arbitrariness of $\eta$, the continuous function $r\mapsto \xi(r)$ must vanish everywhere in $(R_1,R_2)$. That is, for every fixed $r\in (R_1,R_2)$, the following holds: 
\begin{equation*}
    \int_{\mathbb{S}^{N-1}} f(r\theta) \ \Delta_{\mathbb{S}^{N-1}}\psi(\theta) \, dS_\theta=0\quad \text{for all }\psi\in C^\infty(\mathbb{S}^{N-1}).
\end{equation*}
In particular, for every spherical harmonic $Y_k$ (see \cite{atkinson_han2012spherical_harmonics_and_approximations}) of degree $k\ge 1$ we have
\begin{equation*}
    -\lambda_k\int_{\mathbb{S}^{N-1}} f(r\theta) Y_k(\theta)   \, dS_\theta=0\quad \text{for all }k\ge 1,
\end{equation*}
where $\lambda_k$ is the eigenvalue associated with $Y_k$. Since $\lambda_k>0$ for $k\ne 0$, then the function $\theta\mapsto f(r\theta)$ must be orthogonal in $L^2(\mathbb{S}^{N-1})$ to all spherical harmonics $Y_k$ of positive degree. Hence, $\theta\mapsto f(r\theta)$ is a constant function for all $r\in (R_1,R_2)$.
We conclude that $f$ is radial, as claimed.
\end{proof}

\section{Proof of Theorem \ref{mainthm I}}\label{section Proof Thm I}
In this section, we will give a proof of Theorem \ref{mainthm I} by employing the four results introduced in the previous section. As before, we will work in the setting given by \eqref{Omega}--\eqref{g is analytic}.

\begin{proof}[Proof of Theorem \ref{mainthm I}]
\noindent {\bf Step 1: ``$u$ is radial"} 

Let $u\in H_0^1(\Omega)$ denote the solution to the overdetermined problem \eqref{odp}. First, notice that, in either case of \eqref{two possibilities}, $u$ is radial inside $S$.

Since $\Omega=D\cupdot\Gamma\cupdot S$ (where $\cupdot$ denotes the disjoint union) and $\Gamma$ has zero Lebesgue measure by hypothesis,
\eqref{weak form} and \eqref{sigmas} yield the following:
\begin{equation}\label{new_weak form u split}
    \int_D \sigma\nabla u\cdot\nabla \varphi + \int_{S} \nabla u \cdot \nabla \varphi = \int_\Omega g\varphi \quad \text{for all }\varphi\in H^1_0 (\Omega).
\end{equation}

First, we will show that $u$ (and hence $\sigma$) is radial. To this end, let $q$ be the radial function constructed in Lemma \ref{lem radial function}.
By construction, $q$ satisfies
\begin{equation}\label{new_weak form q split}
    \int_D \nabla q\cdot\nabla \varphi + \int_{S} \nabla q \cdot \nabla \varphi = \int_\Omega g\varphi \quad \text{for all }\varphi\in H^1_0(\Omega).    
\end{equation}

In particular, since both $u$ and $q$ are real analytic functions inside the connected open set $\Omega\setminus\overline{D}$ and satisfy the same equation and the same Cauchy data on $\partial \Omega$, the Cauchy--Kovalevskaya theorem \cite[Subsection 4.6.3]{Evans2010PartialDifferentialEquations} and the identity theorem for real analytic functions imply that $u\equiv q$ in $\Omega\setminus\overline{D}$, and thus 
\begin{equation}\label{nabla equal outside}
\nabla u \equiv \nabla q  \quad \text{in } \Omega\setminus\overline{D}.   
\end{equation}
Let now $\widetilde\sigma$ denote a smooth ``extension" of $\sigma$, as constructed in Lemma \ref{lem extension}. Subtracting \eqref{new_weak form q split} from \eqref{new_weak form u split} and taking $\varphi:=\widetilde{\sigma}u - q\in H_0^1(\Omega)$ as a test function yield
\begin{equation}\label{most important identity}
0= \int_D (\sigma\nabla u - \nabla q )\cdot \nabla(\widetilde\sigma u - q) = \int_D |\nabla(\widetilde\sigma u-q)|^2,
\end{equation}
where, in the second equality, we have used the fact that $\nabla(\widetilde\sigma u)= \sigma\nabla u$ in $D$ because $\widetilde\sigma$ is locally constant in $D$.
As a result, $\sigma \nabla u = \nabla(\widetilde\sigma u)= \nabla q$ inside $D$. Combining this with \eqref{nabla equal outside} and recalling that $q$ is a radial function yield that $\nabla_\theta u = 0$ in $\Omega$. Hence, by Lemma \ref{lem radial symmetry}, the function $u$, which is continuous in $\Omega$ by \cite[Theorem 8.22]{GilbargTrudinger}, is radial in each annulus concentric with $\Omega$, and thus in $\Omega$ itself.

\noindent {\bf Step 2: ``either $I=\emptyset$, or $I=\{i\}$ and $(\overline{D_i}, \overline{\Omega})$ are  concentric closed balls"} 

Since $u$ is radial in $\Omega$, then $\sigma$ is also radial in $\Omega\setminus\Gamma$. Indeed, the identity
\begin{equation*}
  \sigma(x)= -1/\Delta u(x)
\end{equation*}
holds true for all $x\in \Omega\setminus\Gamma$ by localizing \eqref{weak form} inside each phase.
If $I=\emptyset$, then there is nothing to prove, and by Lemma \ref{lem: two cases} $\Omega$ is either a ball or a spherical annulus. Otherwise, take an arbitrary index $i\in I$ and a point $x\in D_i\setminus\{0\}$. Since $D_i$ is open, we can find some small enough radius $\varepsilon>0$ such that $B_\varepsilon(x)\subset D_i$. By definition $\sigma\equiv\sigma_i$ inside $B_\varepsilon(x)$. Moreover, since $\sigma$ is radial in $\Omega\setminus\Gamma$, it follows that $\sigma\equiv\sigma_i$ inside $\left(B_{|x|+\varepsilon}\setminus\overline{B_{|x|-\varepsilon}}\right)\setminus\Gamma$. Thus, recalling that $\Gamma$ has empty interior, we get 
\begin{equation}\label{ol D_i inside annulus}
 \overline{B_{|x|+\varepsilon}}\setminus B_{|x|-\varepsilon}\subset  \overline{D_i}.  
\end{equation}
The closed annulus in \eqref{ol D_i inside annulus}
splits $\overline\Om$ into two disjoint parts, only one of which must contain the connected set $\overline{S}$. This rules out case ($ii$) of \eqref{two possibilities}. 
Therefore, $\Omega$ must be a ball, and $\overline{S}$ is entirely contained in the outer part. In other words, 
\begin{equation}\label{S lies outside}
    \overline{S}\cap B_{|x|-\varepsilon}=\emptyset.
\end{equation}
Take now an arbitrary $\theta\in \mathbb{S}^{N-1}$ and set 
\begin{equation*}
    T:=\inf \setbld{t\ge 0}{s\theta\in \overline{D_i}\quad \text{for all }s\in \big[t,|x|\big]}.
\end{equation*}
By the considerations above, $T$ is a well-defined nonnegative real number not greater than $|x|-\varepsilon$. We claim that $T=0$, that is, the open segment (and thus the closed one) connecting $|x|\theta$ to the origin lies entirely in $\overline{D_i}$. Suppose, by way of contradiction, that $T>0$. As a result, $T\theta\in \partial\overline{D_i}$. Thus, there would exist a radius $\eta>0$ such that the open ball
$B_\eta (T\theta)$ contains a point $y$ outside $\overline{D_i}$. In particular, by \eqref{ol D_i inside annulus}, we must have $|y|<|x|-\varepsilon$, whence, by \eqref{S lies outside}, $y$ must belong to some $\overline{D_j}$ with $j\ne i$, in contrast with the discreteness assumption \eqref{discreteness}. We conclude that $T=0$ and thus the segment connecting $x$ to the origin lies in $\overline{D_i}$. Now, since $x\in D_i$ and $\theta\in \mathbb{S}^{N-1}$ were arbitrarily chosen, we have obtained the following result:
\begin{equation}\label{ball in D_i}
  \text{for all } x\in D_i, \quad \overline{B_{|x|}}\subset \overline{D_i}.
\end{equation}
Take now a point $z\in \overline{D_i}$ satisfying $|z|=\max_{\zeta\in \overline{D_i}}|\zeta|$. 
By definition,
\begin{equation}\label{D_i in Ball}
\overline{D_i}\subset \overline{B_{|z|}}.
\end{equation}
On the other hand, we can take a sequence $(z_n)_n$ of points in $D_i$ whose norms $|z_n|$ monotonically converge to $|z|$. 
Then, \eqref{ball in D_i} yields
\begin{equation*}
B_{|z|}=\bigcup_{n\in\mathbb{N}}\overline{B_{|z_n|}}\subset \overline{D_i},
\end{equation*}
whence
\begin{equation}\label{Ball in D_i}
\overline{B_{|z|}}\subset \overline{D_i}.
\end{equation}
By combining \eqref{D_i in Ball}--\eqref{Ball in D_i}, we get that $\overline{D_i}$ is a closed ball centered at the origin. Moreover, since the index $i\in I$ was arbitrary, $I$ must contain at most one element, in other words, $I=\{i\}$. This concludes the proof. 
\end{proof}

\section{Some applications to parabolic overdetermined problems}\label{section Parabolic}

Throughout this section, we will work in the setting given by \eqref{Omega}--\eqref{discreteness} and show how Theorem \ref{mainthm I} can be applied to obtain symmetry in overdetermined problems of parabolic~type.  

We start by considering problem \eqref{parabolic_bvp} under the assumption \eqref{Sigma} and the following Serrin-type overdetermined condition:
\begin{equation}\label{parabolic_Serrin-type condition}
    \partial_\nu u (x,t)= c(t) \quad \text{on } \Sigma \times (0,\infty),
\end{equation}
where $c: (0,\infty) \to \mathbb{R}$ is a function of $t$ only, $\Sigma\subset \partial\Omega$ is a sphere (see \eqref{Sigma}), and $\partial_\nu$ denotes the normal derivative computed with respect to the outward unit normal $\nu$.
The condition in \eqref{parabolic_Serrin-type condition} can be described by saying that $\Si$ is a surface with the {\it constant flow property}, in the sense that the heat flux across $\Si$ depends only on time (see, for instance, \cite{savo2016heat,savo2018geometric,CavallinaMagnaniniSakaguchi2021constantflowproperty} for details). We remark that \eqref{parabolic_Serrin-type condition} can be interpreted in the classical sense, that is, pointwise, thanks to the boundary regularity theory (see \cite[Theorem 5.2, Chapter IV Section 5
 ]{ladyzhenskaja1968linear}). 

\begin{corollary}\label{cor_parabolic_1}
In the setting given by \eqref{Omega}--\eqref{Sigma}, the following are equivalent.
\begin{enumerate}[(i)]
    \item The overdetermined problem \eqref{parabolic_bvp}--\eqref{parabolic_Serrin-type condition} admits a weak solution $u$ (see  Definition~\ref{definition}). 
    \item 
    \begin{itemize}
        \item Either $I=\emptyset$ and $\Omega$ is either a ball or a spherical annulus (one-phase), 
        \item or $I=\{i\}$ and $(\overline{D}, \overline{\Omega})=(\overline{D_i}, \overline{\Omega})$ are  concentric closed balls (two-phase).
    \end{itemize}
\end{enumerate}
\end{corollary}
\begin{proof}
Clearly $(ii) \implies (i)$, so we only need to prove $(i) \implies (ii)$. Let $u$ be solution of \eqref{parabolic_bvp}--\eqref{parabolic_Serrin-type condition}. Following \cite{sakaguchi2016rendiconti}, we consider
\begin{equation}
    v(x)= \int_0^\infty u(x,t) dt .
\end{equation}
As one can easily check by formal computation with \eqref{parabolic_bvp} at hand, $v$ satisfies \eqref{odp} with $g\equiv 1$ (see Appendix \ref{Appendix} for a rigorous proof of this fact). Thus, the result easily follows from Theorem~\ref{mainthm I}.
\end{proof}

We are now ready to provide the proof of Theorem~\ref{thm:parabolic}.
\begin{proof}[Proof of Theorem \ref{thm:parabolic}]
Clearly $(iii) \implies (i)$ and $(iii) \implies (ii)$. We are going to prove $(i) \implies (iii)$, and $(ii) \implies (iii)$.

$(i) \implies (iii)$. Reasoning as in \cite{sakaguchi2016rendiconti, sakaguchibessatsu} (see, for instance, \cite[first paragraph of Section~3]{sakaguchi2016rendiconti}), we get that $u$ is radial in $\overline{\Omega} \setminus D$. Without loss of generality, we assume that the center of symmetry is the origin. Given the boundary condition in \eqref{parabolic_bvp}, we arrive at the two cases in \eqref{two possibilities}.
Thus, we have that $u$ solves the overdetermined problem \eqref{parabolic_bvp}--\eqref{parabolic_Serrin-type condition} and the desired conclusion follows from Corollary \ref{cor_parabolic_1}.

$(ii) \implies (iii)$. Reasoning as in \cite[Second and third paragraphs of Subsection 4.2]{CavallinaMagnaniniSakaguchi2021constantflowproperty}, we obtain, as before, that $u$ is radial in $\overline{\Omega} \setminus D$. The rest of the proof is identical to that of the implication $(i) \implies (iii)$. 
\end{proof}

\appendix
\section{Integrating the solution of the multi-phase heat equation}\label{Appendix}

In what follows, let $\Omega$ be a bounded domain in $\mathbb{R}^N$ ($N\ge 2$).
In this appendix, we will show how to rigorously justify the formal integration described in the proof of Corollary~\ref{cor_parabolic_1}. 

In what follows, we will handle functions that depend on both the time and space variables. In this setting, it will be helpful to make use of the ``Bochner space approach" to untangle the dependencies on each variable. For instance, a function $f=f(x,t)$ will be identified with the mapping $t\mapsto f(t)$, where $f(t)(\cdottone)$ is an element of some function space (which will be specified at each occurrence) and $x\mapsto f(t)(x):=f(x,t)$.

We recall the definition of weak solution to the initial/boundary-value problem \eqref{parabolic_bvp} (see \cite[Definition of Subsection 7.1.1]{Evans2010PartialDifferentialEquations}).
\begin{definition}[Weak solution]\label{definition} 
We say that a function $u$ is a weak solution to \eqref{parabolic_bvp} if, for any positive $T>0$,  
\begin{equation}\label{conditions for an admissible weak solution}
    u\in L^2 \left((0, T ),H_0^1(\Omega)\right), \quad \partial_t u\in L^2\left((0,T), H^{-1}(\Omega)\right)
\end{equation}
and the following holds for each $\varphi\in H_0^1(\Omega)$ and almost every $t\in [0,\infty)$:
\begin{equation}\label{weak form parabolic eq}
 \left\langle \partial_t u(t), \varphi\right\rangle  = - \int_\Omega \sigma\nabla u(t)\cdot \nabla\varphi\, dx  
\end{equation}
and 
\begin{equation}\label{app: initial condition}
    u(0)\equiv 1.
\end{equation}
Here and in what follows, we make use of the notation $\langle\cdottone,\cdottone\rangle$ to denote the duality pairing between $H^{-1}(\Omega):=\left(H_0^1(\Omega)\right)^*$ and $H_0^1(\Omega)$.
The time derivative $\partial_t u$ in the expression above has to be understood in a suitable weak sense. Namely, $\partial_t u\in L^2\left((0,T), H^{-1}(\Omega)\right)$ is the function such that the following identity between Bochner integrals holds for all $\psi\in C_c^\infty(0,T)$:
\begin{equation}\label{weak time derivative}
 \int_0^T \partial_t u(t) \psi(t)\, dt = - \int_0^T u(t)\psi'(t)\, dt.   
\end{equation}
\end{definition}
We remark that \eqref{app: initial condition} makes sense since $u\in C\left([0,T], L^2(\Omega)\right)$ by \cite[Theorem~3 in Subsection 5.9.2]{Evans2010PartialDifferentialEquations}.

From Definition \ref{definition}, it is not difficult to obtain the following decay estimate for the $L^2(\Om)$-norm of $u$.
\begin{lemma}
\begin{equation}\label{eq:decay estimate for L^2 norm}
    \int_{\Omega} u(T)^2 \, dx \le | \Omega | \, e^{- 2\lambda T} \quad \text{ for any } T\ge 0 ,
\end{equation}
where 
\begin{equation*}
\lambda := \inf_{w\in H_0^1(\Omega)\setminus\{0\}} \frac{\int_\Omega \sigma |\nabla w|^2 \, dx}{\int_\Omega w^2 \, dx}>0
\end{equation*}
is the first eigenvalue of the operator $-\Delta_\sigma:= \dv\left(\sigma\nabla\cdottone\right)$ with Dirichlet boundary condition in $\Omega$.    
\end{lemma}
\begin{proof}
For any fixed $t > 0$, taking $\varphi:= u(t)$ as a test function in \eqref{weak form parabolic eq} we find that
\begin{equation*}
   \begin{split}
       0 
       & = 2 \left[  \langle \partial_t u(t) , u(t) \rangle + \int_{\Omega} \sigma |\nabla u(t)|^2 \,dx \right]
       \\
       & \ge  \partial_t \left[ \int_{\Omega} u(t)^2 \, dx \right]  + 2\lambda \int_{\Omega} u(t)^2 \, dx ,
   \end{split} 
\end{equation*}
where in the last inequality we used that $2 \langle \pa_t u(t) , u(t) \rangle = \partial_t \left( \int_{\Omega} u(t)^2 \, dx\right)$ (see \cite[Theorem~3 of Subsection 5.9.2]{Evans2010PartialDifferentialEquations}) and the definition of $\lambda$. Multiplying by $e^{2\lambda t}$, the last formula gives that
\begin{equation*}
    0 \ge \partial_t \left[ \left( \int_{\Omega} u(t)^2 \, dx \right) \, e^{2\lambda t}  \right] ,
\end{equation*}
which integrated with respect to $t$ on $(0,T)$ gives
\begin{equation*}
 0 \ge \left( \int_{\Omega} u(T)^2  \, dx \right) \, e^{2\lambda T} - \int_{\Omega} u(0)^2 \, dx ,   
\end{equation*}
from which \eqref{eq:decay estimate for L^2 norm} follows by recalling that $u(0) \equiv 1$ by the initial condition in \eqref{parabolic_bvp}.
\end{proof}

\begin{theorem}
    If $u$ is a weak solution to \eqref{parabolic_bvp}, then $t\mapsto u(t)\in H_0^1(\Omega)$ is Bochner integrable on $(0,\infty)$~and
    \begin{equation}\label{v definition}
    v := \int_0^\infty u(t) \, dt 
    \end{equation}
    is a weak solution to
    \begin{equation}\label{eq:Appendix: two-phase Dirichlet elliptic}
    -\dv(\sigma\nabla v)=1 \quad \text{in }\Omega, \quad v=0\quad \text{on }\partial\Omega .
\end{equation}

\end{theorem}
\begin{proof}
First of all, notice that $t\mapsto u(t)\in L^2(\Omega)$ is Bochner integrable on $(0,\infty)$ because it is measurable and its norm is summable by \eqref{eq:decay estimate for L^2 norm}.
Integrating \eqref{weak form parabolic eq} with respect to $t$ on $(0,T)$ yields
\begin{equation}\label{app: zeroth equation}
    \int_0^T \langle \partial_t u(t),\varphi\rangle\, dt = \int_0^T \left( \int_\Omega \sigma \nabla u(t)\cdot \nabla\varphi   \right)\, dt.
\end{equation}
By using the fact that Bochner integrals commute with the bounded linear operators $\langle \cdottone,\varphi\rangle: H^{-1}(\Omega)\to \mathbb{R}$, $\int_\Omega \sigma\nabla(\cdottone) \cdot \nabla\varphi\, dx: H_0^1(\Omega)\to \mathbb{R}$, and $\nabla: H_0^1(\Omega)\to L^2(\Omega, \mathbb{R}^N)$ and the fundamental theorem of calculus applied to the function $u\in W^{1,2}\left((0,T), H^{-1}(\Omega)\right)$ (see \cite[Theorem 2, (ii) of Subsection 5.9.2]{Evans2010PartialDifferentialEquations}) we get
\begin{equation}\label{app: first equation}
\int_0^T \langle \partial_t u(t),\varphi\rangle\, dt = \left\langle \int_0^T \partial_t u(t) \, dt, \varphi\right\rangle= \langle (u(T)-u(0)), \varphi\rangle = \int_\Omega  (u(T)-u(0))\varphi\, dx
\end{equation}
and 
\begin{equation}\label{app: second equation}
\begin{split}
\int_0^T \left( \int_\Omega \sigma \nabla u(t)\cdot \nabla\varphi   \right)\, dt
& =\int_\Omega\sigma \left(\int_0^T\nabla u(t) \, dt\right) \cdot \nabla\varphi \, dx\\
& = \int_\Omega\sigma \nabla\left(\int_0^T u(t) dt\right) \cdot \nabla\varphi \, dx.
\end{split}
\end{equation}
Combining \eqref{app: zeroth equation}, \eqref{app: first equation}, \eqref{app: second equation}, and recalling that $u(0) \equiv 1$ by the initial condition in~\eqref{parabolic_bvp}, we obtain that
\begin{equation*}
\int_\Omega  ( u(T) - 1 ) \varphi\, dx=    \int_\Omega\sigma \nabla\left(\int_0^T u(t) dt\right) \cdot \nabla\varphi \, dx.
\end{equation*}
Now, since $\varphi\in H_0^1(\Omega)$ is arbitrary, the above implies that the function \begin{equation*}
v_T:=\int_0^T u(t) \, dt\in H_0^1(\Omega)    
\end{equation*}
is a weak solution to 
\begin{equation}\label{some elliptic equation}
    \begin{cases}
        -\dv\left(\sigma\nabla v_T\right)=f\quad \text{in }\Omega,\\
        v_T=0\quad \text{on }\partial\Omega,
    \end{cases}
\end{equation}
for $f=1-u(T)$.
Let now $(-\Delta_\sigma)^{-1}$
denote the ``solution operator" that assigns to any $f\in L^2(\Omega)$ the solution to \eqref{some elliptic equation}. We recall that by the standard Hilbert space approach to elliptic equations (see \cite[Section 8.2]{GilbargTrudinger}), $(-\Delta_\sigma)^{-1}$
is a bounded linear operator from $H^{-1}(\Omega)$ to $H_0^1(\Omega)$ (hence, in particular, from $L^2(\Omega)$ to $H_0^1(\Omega)$).
Finally, by using the fact that the solution $u(t)$ to \eqref{weak form parabolic eq} converges to $0$ in the $L^2(\Omega)$-norm as $t\to\infty$ (see \eqref{eq:decay estimate for L^2 norm}), we obtain that 
\begin{equation*}
 v:=\int_0^\infty u(t) dt = \lim_{T\to\infty} v_T = \lim_{T\to \infty} (-\Delta_\sigma)^{-1}\left(1-u(T)\right)=(-\Delta_\sigma)^{-1} (1),
\end{equation*}
where the limits above hold in the $H_0^1(\Omega)$-norm.
In other words, $v$ solves \eqref{eq:Appendix: two-phase Dirichlet elliptic}, and the first part of the claim is proved.
\end{proof}

\begin{definition}\label{def neumann on Sigma}
Let $\Sigma\subset \partial\Omega$ be a Lipschitz portion of $\partial\Omega$ and let $c(t)\in L^2(\Sigma)$ for $t\in (0,\infty)$.
We say that a function $u\in C\left([0,T], L^2(\Omega)\right)$ for all $T>0$ satisfying \eqref{conditions for an admissible weak solution} is a weak solution to \begin{equation}\label{parabolic eq + Neumann = c(t)}
\begin{cases}
\partial_t u = \dv(\sigma\nabla u) \quad \text{in }\Omega \times (0,\infty),\\
u=1 \quad \text{in } \Omega \times \left\{ 0 \right\} , \\
\partial_\nu u(t)=c(t)\quad \text{on }\Sigma\times (0,\infty)
\end{cases}
\end{equation}
if $u(0)\equiv 1$ and the following holds for almost every $t\in (0,\infty)$, and for all $\varphi\in H_\Sigma^1(\Omega)$:
\begin{equation}\label{neumann weak condition}
 \left\langle \partial_t u(t), \varphi\right\rangle  = - \int_\Omega \sigma\nabla u(t)\cdot \nabla\varphi\, dx + \int_\Sigma c(t)\varphi\, dS. 
\end{equation}
Here $H_\Sigma^1(\Omega)\subset H^1(\Omega)$ is the subset of functions with ``vanishing trace" outside $\Sigma$:
\begin{equation*}
H_\Sigma^1(\Omega):=\overline{\setbld{\varphi\in C^\infty(\Omega)}{\text{$\exists K\subset \Omega\cup \Sigma$ compact: $\restr{\varphi}{\Omega\setminus K}\equiv 0$}}}^{\ \norm{\cdottone}_{H^1(\Omega)}}.
\end{equation*}
\end{definition}
\begin{theorem}
 Let $\Sigma$ be a Lipschitz portion of $\partial\Omega$ and let $u$ be a weak solution to \eqref{parabolic eq + Neumann = c(t)}. 
 Moreover, assume that at least one of the following assumptions holds:
 \begin{enumerate}[(i)]
    \item $c\in L^1\left((0,\infty), L^2(\Sigma)\right)$.
    \item For each $t>0$, $c(t)$ is a constant function on $\Sigma$.
    \item The function $(0,\infty)\times \Sigma \ni (t,x)\mapsto c(t)(x)\in [-\infty,\infty]$ is measurable. Moreover, either $c(t)\le 0$ or $c(t)\ge 0$ holds almost everywhere on $\Sigma$ and for almost all $t>0$.
\end{enumerate}
 
 Then $\int_0^\infty c(t)\, dt$ is a well-defined measurable function on $\Sigma$. Moreover, the function $v$ defined in \eqref{v definition} satisfies
 \begin{equation}\label{weak neumann condition for v}
    \int_\Omega \sigma\nabla v\cdot\nabla\varphi \, dx = \int_\Omega \varphi\, dx + \int_\Sigma \left(\int_0^\infty c(t)\, dt\right)\varphi\, dS \quad \text{for all}\quad \varphi\in H_\Sigma^1(\Omega)
 \end{equation}
and the integrals above are all finite. 
\end{theorem}
\begin{proof}
\noindent {\bf When $(i)$ holds.} 

Fix an arbitrary $\varphi\in H_\Sigma^1(\Omega)$. Then, the map $t\mapsto c(t)\varphi\in L^1(\Sigma)$ is Bochner integrable on $(0,\infty)$ as it is measurable and its norm is summable by the Cauchy--Schwarz inequality:
\begin{equation*}
    \norm{c(t)\varphi}_{L^1(\Sigma)}\le \norm{c(t)}_{L^2(\Sigma)}\norm{\varphi}_{L^2(\Sigma)}. 
\end{equation*}
Thus, we can integrate \eqref{neumann weak condition} to get the desired identity:
\begin{equation*}
    -\int_\Omega \varphi \, dx = -\int_\Omega \sigma\nabla v\cdot \nabla \varphi + \int_0^\infty \left(\int_\Sigma c(t)\varphi\, dS\right)\, dt = 
    -\int_\Omega \sigma\nabla v\cdot \nabla \varphi +  \int_\Sigma\left( \int_0^\infty   c(t)\, dt\right)\varphi\, dS.
\end{equation*}

\noindent {\bf When $(ii)$ holds.} 

Take a test function $\varphi_0$ satisfying
\begin{equation}\label{vaprhi_0}
\varphi_0\in H_\Sigma^1(\Omega),\quad  \restr{\varphi_0}{\Sigma}\equiv 1.    
\end{equation}
Then, \eqref{neumann weak condition} yields the following for almost every $t>0$:
\begin{equation}\label{c(t) equals something}
    c(t) = \frac{1}{|\Sigma|}\left(\langle\partial_t u(t), \varphi_0\rangle+ \int_\Omega \sigma\nabla u(t)\cdot\nabla \varphi_0\, dx    \right).
\end{equation}
In particular, if, by a slight abuse of notation, we regard $t\mapsto c(t)$ as a mapping $(0,\infty)\to \mathbb{R}$, then \eqref{c(t) equals something} implies that $c\in L^1\left((0,\infty), \mathbb{R}\right)$.  
Then, once again, integrating \eqref{neumann weak condition} we get the desired identity:
\begin{equation*}
    \int_\Omega -\varphi\, dx = -\int_\Omega \sigma \nabla v \cdot \nabla \varphi \, dx + \int_0^\infty c(t)\, dt \int_\Sigma \varphi \, dS = 
    -\int_\Omega \sigma \nabla v \cdot \nabla \varphi \, dx +  \int_\Sigma \left(\int_0^\infty c(t)\, dt\right)\varphi \, dS.
\end{equation*}

\noindent {\bf When $(iii)$ holds.} 

Without loss of generality, let $c(t)\ge 0$ almost everywhere on $\Sigma$ for almost all $t\in (0,\infty)$. 
Notice that, by taking $\varphi_0$ (defined as in \eqref{vaprhi_0}) as a test function in \eqref{neumann weak condition} and integrating on $(0,\infty)$, we obtain 
\begin{equation*}
-\int_\Omega \varphi_0\, dx + \int_\Omega \sigma \nabla v\cdot \nabla \varphi_0 = \int_0^\infty \left(\int_\Sigma c(t)\, dS\right) \, dt = \int_0^\infty \norm{c(t)}_{L^1(\Sigma)}\, dt.
\end{equation*}
It follows that $c\in L^1\left((0,\infty), L^1(\Sigma)\right)$, as $t\mapsto c(t)\in L^1(\Sigma)$ is measurable and its $L^1(\Sigma)$-norm is summable on $(0,\infty)$. In particular, $\int_0^\infty c(t)\, dt$ defines a measurable function on $\Sigma$, as claimed. 

Then, take an arbitrary test function $\varphi\in H_\Sigma^1(\Omega)$ and consider its positive and negative parts $\varphi_\pm\in H_\Sigma^1(\Omega)$. We can integrate \eqref{neumann weak condition} on $(0,\infty)$ and apply Tonelli's theorem (see \cite[Theorem 2.37, a.]{folland1999real_analysis}) to obtain the following:
\begin{equation}\label{some integral identity in case (iii)}
    \int_\Omega \varphi_\pm \, dx + \int_\Omega \sigma \nabla v \cdot \nabla \varphi_\pm \, dx = \int_0^\infty \left(\int_\Sigma c(t) \varphi_\pm \, dS\right) dt= \int_\Sigma \left( \int_0^\infty c(t)\, dt\right) \varphi_\pm \, dS.
\end{equation}
Also, notice that the integral on the right-hand side above must be finite, as the ones on the left-hand side are. Finally, since $\varphi=\varphi_+-\varphi_-$, we recover \eqref{weak neumann condition for v} by linearity.  
\end{proof}

\begin{remark}
Notice that, the assumptions $(i), (ii), (iii)$ are all satisfied if, in particular, $\Sigma\subset\partial \Omega$ is a sphere, $c(t)$ is a constant function for all $t$ and $u$ is the solution to the Dirichlet problem (that is, as in the setting of Section \ref{section Parabolic}).
\end{remark}

\section*{Data availability statement}
The paper has no associated data.

\section*{Acknowledgements}
Lorenzo Cavallina was partially supported by JSPS KAKENHI Grant Number JP22K13935 and JP21KK0044.

Giorgio Poggesi is supported by the Australian Research Council (ARC) Discovery Early Career Researcher Award (DECRA) DE230100954 “Partial Differential Equations: geometric aspects and applications” and is a member of the Australian Mathematical Society (AustMS).

\bibliographystyle{siam}
\bibliography{references}
\bigskip

\noindent
\textsc{Lorenzo Cavallina:}\\
\noindent
\textsc{
Mathematical Institute, Tohoku University, Aoba-ku, 
Sendai 980-8578, Japan}\\
\noindent
{\em Electronic mail address:}
cavallina.lorenzo.e6@tohoku.ac.jp

\bigskip

\noindent
\textsc{Giorgio Poggesi:}\\
\noindent
\textsc{
Discipline of Mathematical Sciences, The University of Adelaide, Adelaide SA 5005, Australia}
\\
\noindent
{\em Electronic mail address:}
giorgio.poggesi@adelaide.edu.au

%
%%%%%%%%%%%%%%%%%%%%%%%%%%%%%%%%%%%%%%%%%%%%%%%%%%%%%%%%
%\\
%\noindent
%\textsc{and}
%\\
%\textsc{
%Department of Mathematics and Statistics, The University of Western Australia, %35 Stirling Highway, Crawley, Perth, WA 6009, Australia} \\
%\noindent
%{\em Electronic mail address:}
%giorgio.poggesi@uwa.edu.au
%%%%%%%%%%%%%%%%%%%%%%%%%%%%%%%%%%%%%
%

\end{document}